\documentclass[11pt]{article}       
\usepackage{geometry}               
\usepackage{graphicx}
\usepackage{caption}
\usepackage{subcaption}
\usepackage{enumerate}
\usepackage{comment}
\usepackage{amsmath} % assumes amsmath package installed
\usepackage{amssymb}  % assumes amsmath package installed
\usepackage{amsthm}  % assumes amsmath package installed
\usepackage{bm}
\usepackage{lipsum}
\usepackage{hyperref}
\usepackage{color}

\newtheorem{proposition}{Proposition}
\newtheorem{definition}{Definition}

\newtheorem{lemma}{Lemma}
\newtheorem{theorem}{Theorem}
\newtheorem{remark}{Remark}
\newtheorem{assumption}{Assumption}

\numberwithin{equation}{section}

\title{A Dynamic Game Approach to Distributionally Robust  Safety Specifications for Stochastic Systems} 

\author{
 Insoon Yang\thanks{Department of Electrical Engineering, University of Southern California, Los Angeles, CA 90089, USA (insoonya@usc.edu). }%Supported in part by NSF under CPS:FORCES (CNS1239166).}
}

\date{}
\providecommand{\keywords}[1]{\textbf{Key words.} #1}

\begin{document}
\maketitle

%\footnotetext[1]{Department of Mathematics, University of California, Berkeley
%({miller@math.berkeley.edu}). Supported in part by NSF GRFP under grant number DGE 1106400.}
%\footnotetext[2]{Laboratory for Information and Decision Systems, Massachusetts Institute of Technology ({insoon@mit.edu}). Supported in part by NSF CPS project FORCES under grant number 1239166.}

\pagestyle{myheadings}
\thispagestyle{plain}

\begin{abstract}
This paper presents a new safety specification method that is robust against errors in the probability distribution of disturbances.
Our proposed distributionally robust safe policy maximizes the probability of a system remaining in a desired set for all times, subject to the worst possible disturbance distribution in an ambiguity set.
We propose a dynamic game formulation of constructing such policies 
%Extending the semi-continuous minimax stochastic control model of Gonz\'{a}lez-Trejo et al. \cite{Gonzalez2003} to multiplicative dynamic programming, 
and identify conditions under which a non-randomized Markov  policy is optimal. 
Based on this existence result, we develop a practical design approach to safety-oriented stochastic controllers with limited information about disturbance distributions.
This control method can be used to minimize another cost function while
ensuring safety in a probabilistic way. %with ambiguous information about disturbance distributions.
However, an associated Bellman equation involves infinite-dimensional minimax optimization problems since the disturbance distribution may have a continuous density.
To resolve computational issues, we propose a duality-based  reformulation method that converts the infinite-dimensional minimax problem into a semi-infinite program that can be solved using existing convergent algorithms. 
We prove that there is no duality gap, and that this approach thus preserves optimality. 
The results of numerical tests confirm that the proposed method is robust against distributional errors in disturbances, while a standard stochastic safety specification tool is not.
\end{abstract}


\section{Introduction} \label{intro}

Various critical decision-making and control problems associated with engineering and socio-technical systems are subject to uncertainties. Large-scale data collected from the Internet of Things and cyber-physical systems can provide information about the probability distribution of these uncertainties. 
Statistical learning and filtering methods also support the construction of data-driven distribution models of uncertainty based on the observed data.
Such distributional information can be used to dramatically improve the performance of closed-loop systems if they adopt appropriate controllers, that reduce the conservativeness of classical techniques, such as robust control. 
Several concerns have been raised about how best to incorporate the collected data into critical control and decision-making problems. These concerns center on safety, risk, robustness and reliability because the data and statistical models extracted from the data often result in inaccurate distributional information. 
Among them, we focus on the safety issue and develop a safety specification and management tool with ambiguous information about the probability distribution of disturbances.

For safety-critical systems subject to uncertain disturbances, reachability-based safety specification techniques have been used to compute the reachable sets and safe sets, which allow one to verify that a system is evolving within a safe range of operation and to synthesize controllers to satisfy safety constraints (e.g., \cite{Bertsekas1971}, \cite{Tomlin1998}, \cite{Lygeros1999}, \cite{Cardaliaguet1999}, \cite{Kurzhanski2002}, \cite{Prajna2004}, \cite{Girard2005}, \cite{Rakovic2006}, \cite{Althoff2011}, \cite{Margellos2011}, \cite{Ghaemi2014}, \cite{Chen2016}).
These methods assume that disturbances lie in a compact set, and thus  require information only about the support of disturbances. 
However, these techniques often produce conservative results as no additional information about the probability distribution of uncertain disturbances is used. These deterministic methods are a natural choice when the data of disturbances are not continuously collected, and thus a reliable stochastic model is not available for them. 
Advances in sensing, communication, and computing technologies as well as statistical learning and estimation tools make it possible to shift this paradigm; sensors, data storage and computing infrastructure in various systems can now provide data to help estimate the probability distribution of disturbances. 
Stochastic reachability analysis tools are based on the assumption that the full probability distribution of disturbances is available 
and can often be used to reduce the conservativeness of deterministic safe set computations.
However, this assumption is often restrictive in practice because obtaining an accurate distribution requires large-scale high-resolution sensor measurements over a long training period or multiple periods.
Furthermore, the accuracy of the  distribution obtained with computational methods
 is often unreliable as it is subject to the quality of the observations, statistical learning or filtering methods, and prior knowledge about the disturbances.
Thus, probabilistic safety specification tools can be misleading and lead to the design of an unreliable controller that may violate safety constraints when distributional information about disturbances is inaccurate.

\begin{figure}
\begin{center}
\includegraphics[width=2.1in]{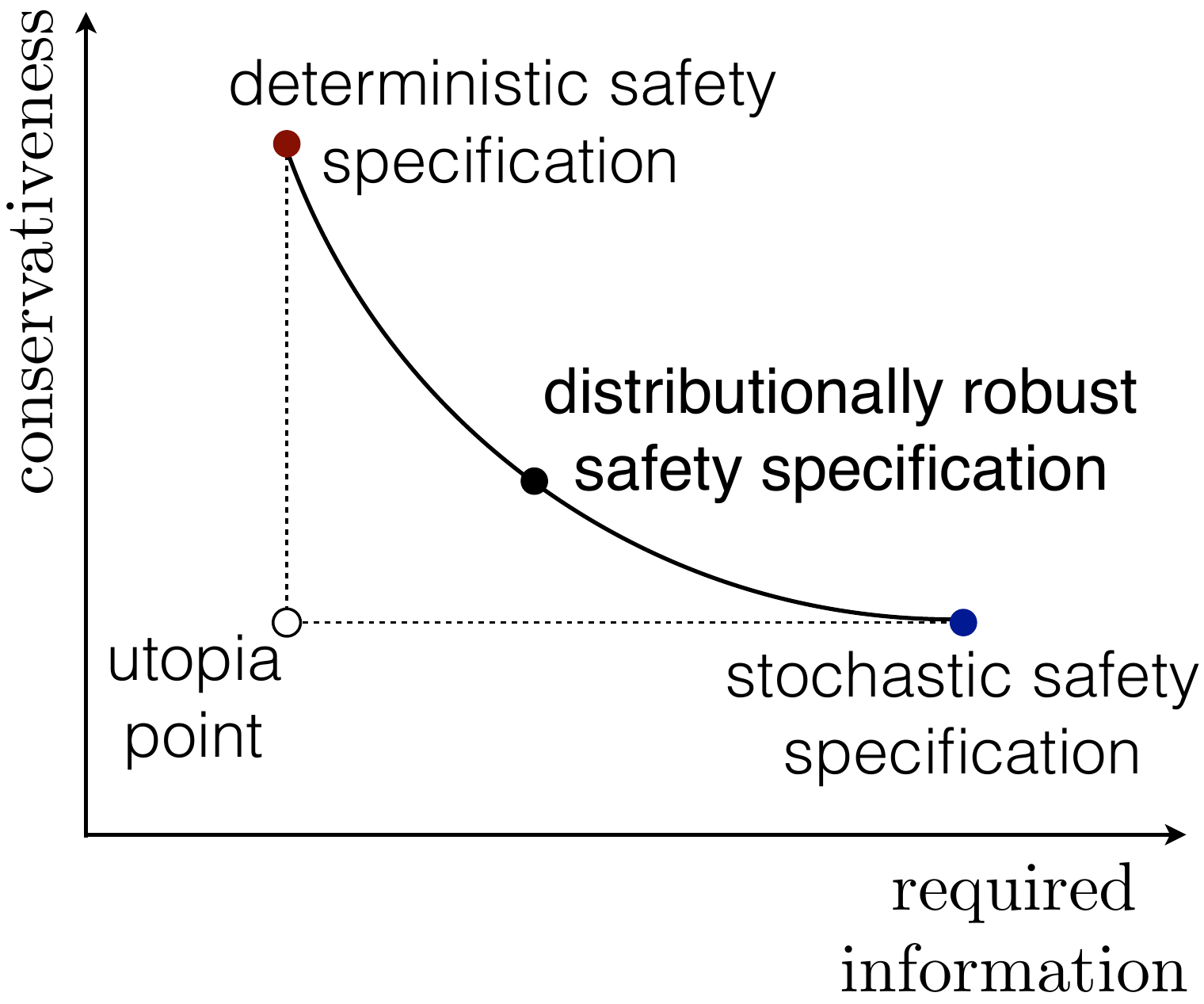}    % The printed column  
\caption{Tradeoff between required information and conservativeness.}  
\label{fig:trade}                                 % Size the figures 
\end{center}                                 % accordingly.
\end{figure}

Fig. \ref{fig:trade} illustrates the tradeoff between required information and conservativeness in deterministic and stochastic safety specification methods.
This study aims to bridge the gap between the two methods by proposing a \emph{distributionally robust safety specification} tool. 
Our approach assumes that the  distribution of disturbances is not fully known but lies in a so-called \emph{ambiguity set} of probability distributions.
If, for example, only the support, mean, and variance of an uncertain variable are reliably estimated, the ambiguity set can be chosen so as to contain distributions consistent with the empirical estimates.
The proposed \emph{distrbutionally robust safe policy} maximizes the probability of a system remaining within a desired set for all times subject to the worst possible disturbance distribution in the ambiguity set.
Therefore, the probabilistic safe set of the closed-loop system is robust against distributional errors within the ambiguity set.

This paper proposes a dynamic game formulation of constructing distributionally robust safe policies and safe sets.
Specifically, it is a two-player zero-sum dynamic game in which Player I selects a policy by which the controller can maximize the probability of safety, while (fictitious) Player II adversarially determines a strategy for the probability distribution of disturbances to minimize the same probability. 
Player II's action space is generally infinite dimensional since the disturbances may have a continuous density function. 
Therefore, the Bellman equation for this dynamic game problem involves infinite-dimensional optimization problems that are computationally challenging. 
Furthermore, the existence of a distributionally robust safe policy is not guaranteed. 
%Extending the semi-continuous minimax control framework of Gonz\'{a}lez-Trejo et al. \cite{Gonzalez2003} to multiplicative dynamic programming, 

The contributions of this work are threefold. 
First, we characterize conditions under which a non-randomized Markov policy is optimal for Player I (controller).
This characterization helps greatly reduce the control strategy space we need to search for because it is enough to restrict our attention to non-randomized Markov policies.
Furthermore, the existence of a non-randomized Markov policy guarantees that the outer maximization problem (for Player I) of an associated Bellman equation is solvable---that is, it is feasible and has an optimal solution. 
Second, we develop a design approach to 
a safety-oriented stochastic controller with limited information about disturbance distributions.
This control method can be used to minimize another cost function while
guaranteeing that the probability for a system being safe for all remaining stages is greater than or equal to a pre-specified threshold, regardless of how the disturbance distribution is chosen in an ambiguity set.
%Furthermore, the existence of a non-randomized Markov policy guarantees that the outer maximization problem (for Player I) of an associated Bellman equation is solvable---that is, it is feasible and has an optimal solution. 
Third, we propose a duality-based reformulation method for the Bellman equation in cases with moment uncertainty.
We show that  there is no duality gap in 
the inner minimization problem (for Player II) of the Bellman equation, which is an infinite-dimensional optimization problem.
Using the strong duality result, we reformulate each infinite-dimensional minimax problem in the Bellman equation 
 as a semi-infinite program
without sacrificing optimality. 
This reformulation alleviates the computational issue arising from the infinite dimensionality of the original Bellman equation because the reformulated Bellman equation can be solved through the use of existing convergent algorithms for semi-infinite programs.

\subsection{Related Work}

A probabilistic reachability tool for stochastic differential equations with jumps has been proposed; it uses a Markov chain approximation to propagate the transition probabilities of the Markov chain backward in time starting from a target set \cite{Hu2005}, \cite{Prandini2006}, \cite{Prandini2006b}. In \cite{Prajna2007}, barrier certificates are employed to calculate an upper bound of the probability that a system will reach a target set. Additionally, \cite{Mitchell2005b} proposes a toolbox that supports expectation-based reachability problems associated with a class of continuous-time stochastic (hybrid) systems by extending the celebrated Hamilton--Jacobi--Isaacs reachability analysis \cite{Tomlin2000}, \cite{Mitchell2005}. 
A partial differential equation characterization of continuous-time stochastic reach-avoid problems is studied in \cite{Esfahani2016} based on the theory of discontinuous viscosity solutions.
For discrete-time stochastic hybrid systems, an elegant dynamic programming approach has been proposed to compute the maximal probability of safety \cite{Abate2008}. This method has been extended to stochastic reach--avoid problems \cite{Summers2010}, stochastic hybrid games \cite{Ding2013}, and partially observable stochastic hybrid systems \cite{Lesser2014}, \cite{Lesser2016}. However, all the aforementioned methods are based on the possibly restrictive assumption that the probability distribution of disturbances is completely known.

This work also closely relates to \emph{distributionally robust control}, which is an emerging stochastic control method.
This method is based on single-stage distributionally robust stochastic optimization that minimizes the worst-case cost, assuming that the probability distribution of uncertain variables lies within an ambiguity set of distributions (e.g., \cite{Scarf1958}, \cite{Dupacova1987}, \cite{ElGhaoui2003}, \cite{Calafiore2006}, \cite{Delage2010}, \cite{Wiesemann2014}). 
For multi-stage problems, a distributionally robust Markov decision process (MDP) formulation has recently been developed while focusing on finite-state, finite-action MDPs \cite{Xu2012}, \cite{Yu2016}.
For cases with moment uncertainty, \cite{VanParys2016} investigates  linear feedback strategies in linear-quadratic settings with risk constraints and proposes a semidefinite programming approach.
We extend the theory of distributionally robust control to the case of continuous state spaces and apply it to reachability analysis and safety specifications.

\subsection{Organization}

The remainder of this paper is organized as follows. 
In Section \ref{setting}, we define distributionally robust safe sets and policies and introduce a dynamic game formulation to construct them.
Section \ref{dp_sol} contains a dynamic programming solution to the problem of computing optimal policies.  
In particular, we  characterize conditions under which a Markov control policy is optimal, and propose a safety-oriented controller design method.
In Section \ref{dual_Bellman}, we consider ambiguity sets with moment uncertainty and show that an associated Bellman equation can be formulated as a semi-infinite program. %using infinite-dimensional duality theory.
An application of the proposed safety specification tool is illustrated through examples in Section \ref{ex}.

\subsection{Notation}

Given a Borel space $X$, $\mathcal{B}(X)$ represents its Borel $\sigma$-algebra.
Given Borel spaces $X$ and $Y$, $Q(A | y)$ denotes
a transition probability (or a stochastic kernel) from $Y$ to $X$, where $Q(\cdot | y)$ is a probability measure on $\mathcal{B}(X)$ for each $y \in Y$ and $Q(A | \cdot)$ is a measurable function on $Y$ for each $A \in \mathcal{B}(X)$. 
Also, $M(X)$ denotes the Banach space of finite signed measures  on $\mathcal{B}(X)$ and $M_+(X)$ represents its positive cone.
For simplicity, we use the following notation of time indexes: $\mathcal{T}:= \{0, 1, \cdots, T-1\}$ and $\bar{\mathcal{T}} := \{0, 1, \cdots, T \}$.

\section{Distributionally Robust Safe Sets and Policies}\label{setting}

\subsection{Stochastic Systems with Distributional Ambiguity}
Consider a discrete-time stochastic system of the form
\begin{equation}\label{sys}
x_{t+1} = f(x_t, u_t, w_t) \quad t \in \mathcal{T}, \quad x_0 = \bm{x},
\end{equation}
where $x_t \in \mathbb{R}^n$ is the state, $u_t \in \mathbb{R}^m$ is the control input, $w_t \in \mathbb{R}^l$ is the stochastic disturbance, and $f: \mathbb{R}^n \times \mathbb{R}^m \times \mathbb{R}^l \to \mathbb{R}^n$ is a measurable function. We assume that the disturbance process $\{w_t\}_{t=0}^{T-1}$ is  defined on a probability space $(\Omega, \mathcal{F}, \mathbb{P})$, and that $w_s$ and $w_t$ are independent for any $s \neq t$.
As mentioned in Section~\ref{intro}, it is often difficult to obtain full information about the probability distribution $\mu_t$ of $w_t$. 
In many cases, only estimates of its support, mean, and variance are available. Furthermore, such estimates are rarely accurate. 
This issue of imperfect disturbance distributions degrades the practicality of classical stochastic reachability tools.
To resolve this problem, we  allow errors in such distributional information and quantify the probability of safety with the worst-case disturbance distribution. 
To mathematically model distributional ambiguity, we assume that the true probability distribution $\mu_t$ of $w_t$ is contained in a so-called \emph{ambiguity set} of distributions, denoted by $\mathbb{D}_t$.
Note that $\mathbb{D}_t$ is a subset of the space $M_+(\mathbb{R}^l)$ of signed measures and thus is infinite dimensional.
An example of such ambiguity sets can be found in Section \ref{amb_moment}. We assume that the  set $\mathbb{D}_t$ of distributions is not empty for each~$t$.

We now briefly discuss admissible control and disturbance distribution strategies. 
Let $H_t$ be the set of histories up to stage $t$, whose element takes the form $h_t = (x_0, u_0, \mu_0, \cdots, x_{t-1}, u_{t-1}, \mu_{t-1}, x_t)$. 
The set of admissible control strategies is chosen as $\Pi := \{ \pi = (\pi_0, \cdots, \pi_{T-1}) \: | \: \pi_t(\mathbb{U}(x_t) | h_t) = 1 \; \forall h_t \in H_t \}$, where $\pi_t$ is a stochastic kernel from $H_t$ to $\mathbb{R}^m$ and $\mathbb{U}(x_t)$ is the set of admissible actions given state $x_t$.
Note that this strategy space is sufficiently broad to contain randomized non-Markov policies. We assume that there exists a measurable function $\pi : \mathbb{R}^n  \to \mathbb{R}^m$ such that $\pi(\bm{x}) \in \mathbb{U}(\bm{x})$ for all $\bm{x} \in \mathbb{R}^n$. 
By viewing the disturbance as an adversarial player who chooses the disturbance's probability distribution based on the available information, the set of admissible disturbance distribution strategies is similarly defined as
$\Gamma := \{ \gamma = (\gamma_0, \cdots, \gamma_{T-1}) | \gamma_t(\mathbb{D}_t | h_t^e) = 1 \; \forall h_t^e \in H_t^e\}$, where $H_t^e$ is an extended set of histories up to stage $t$, whose element is of the form
$h_t^e = (x_0, u_0, \mu_0, \cdots, x_{t-1}, u_{t-1}, \mu_{t-1}, x_t, u_t)$.  Note that the distributional constraints in the ambiguity set $\mathbb{D}_t$ is encoded in this strategy space.
The disturbance (or the adversarial player) can use slightly more information than the controller; the disturbance is aware of the controller's action at stage $t$, $u_t$, in addition to the history $h_t$.
However, the controller cannot be aware of the disturbance distribution's realization $\mu_t$ when making a decision at stage $t$.

\subsection{Distributionally Robust Safety Specifications}

Our goal is to develop a probabilistic safety specification tool that is robust against errors in the probability distribution of the disturbance $\{w_t\}$. Specifically, we compute the worst-case probability of a system remaining in a desired set for all times  when the  distribution of $w_t$ is not fully known but lies within an ambiguity set, $\mathbb{D}_t$.
The proposed \emph{distributionally robust} safety specification tool will be used to design a controller for safety-critical stochastic systems under imperfect information about disturbance distributions.

To formulate a concrete safety specification problem, 
we consider a desired set $A$ for safety, which is an arbitrary compact Borel set in the state space $\mathbb{R}^n$.
We also introduce the following definition of a probabilistic safe set:
\begin{definition}[Probabilistic Safe Set]
We define the probability that the system \eqref{sys} is safe for all $t \in \bar{\mathcal{T}}$
given the strategy pair $(\pi, \gamma)$ and the initial value $\bm{x}$ as
\begin{equation}
P_{\bm{x}}^{\mbox{\tiny \emph{safe}}} (\pi, \gamma; A) :=
\mathbb{P}^{\pi, \gamma} \{x_t \in A \;\; \forall t \in \bar{\mathcal{T}}, \;\; x_0 = \bm{x} \},
\end{equation}
which we call  the \emph{probability of safety} for the set $A$. 
We also define the \emph{probabilistic safe set} with probability $\alpha$ 
under $(\pi, \gamma)$ as the set
\begin{equation}\nonumber
S_\alpha (\pi, \gamma; A) := \{ \bm{x} \in \mathbb{R}^n |  P_{\bm{x}}^{\mbox{\tiny \emph{safe}}} (\pi, \gamma; A) \geq \alpha \}.
\end{equation}
\end{definition}
This set contains all the initial states such that the probability that the system stays in the  set $A$ is greater than or equal to $\alpha$ given the strategy pair $(\pi, \gamma)$.
This definition generalizes the probabilistic safe set introduced in Abate et al. \cite{Abate2008} to the case with ambiguous disturbance distributions.
Using these notions, we now define a distributionally robust safe policy and set as follows:
\begin{definition}[Distributionally Robust Safe  Set]
A control strategy $\pi^\star \in \Pi$ is said to be a \emph{distributionally robust safe policy} given $x_0 = \bm{x}$ if it satisfies 
\begin{equation} \nonumber
\inf_{\gamma \in \Gamma} P_{\bm{x}}^{\mbox{\tiny \emph{safe}}} (\pi^\star, \gamma; A)
\geq \inf_{\gamma' \in \Gamma} P_{\bm{x}}^{\mbox{\tiny \emph{safe}}} (\pi, \gamma'; A)\quad \forall \pi \in \Pi.
\end{equation}
The set $S_\alpha^\star(A)$ is said to be the \emph{distributionally robust safe set} for $A$ with probability $\alpha$  if 
\[
S_\alpha^\star(A) = \{\bm{x} \in \mathbb{R}^n | 
\sup_{\pi \in \Pi} \inf_{\gamma \in \Gamma} P_{\bm{x}}^{\mbox{\tiny \emph{safe}}} (\pi, \gamma; A) \geq \alpha \}.
\]
\end{definition}
In other words, the distributionally robust safe policy $\pi^\star$ maximizes the worst-case probability of safety under distributional ambiguity characterized by the constraints in the  set $\mathbb{D}_t$.
No matter what form the strategy $\gamma$ takes so that the realized distribution lies in the ambiguity set $\mathbb{D}_t$, the probability that the system starting from $\bm{x} \in S_\alpha^*(A)$ stays safe is greater than or equal to $\alpha$ under the distributionally robust safe policy $\pi^\star$.
Once we obtain $\pi^\star$ and $P_{\bm{x}}^{\mbox{\tiny safe}}$ for each $\bm{x}$, we can calculate the distributionally robust safe set $S_\alpha^\star(A)$ through  simple thresholding.
To this end, in the next subsection, we consider a game theoretic formulation to construct a distributionally robust safe policy.

\subsection{A Dynamic Game Formulation}

Let $\bold{1}_A : \mathbb{R}^n  \to \{0, 1\}$ be the indicator function of the set $A \subseteq \mathbb{R}^n$ such that $\bold{1}_A(\bm{x}) = 1$ if $\bm{x} \in A$ and $\bold{1}_A(\bm{x}) = 0$ otherwise.
Then, given $x_0 = \bm{x}$, we have
\begin{equation} \label{prob_safe}
P_{\bm{x}}^{\mbox{\tiny safe}} (\pi, \gamma; A) =  \mathbb{E}^{\pi, \gamma} \bigg [
\prod_{t=0}^T \bold{1}_A(x_t)
 \bigg ],
\end{equation}
where $\mathbb{E}^{\pi, \gamma}$ is the expectation taken with respect to the probability measure $\mathbb{P}^{\pi, \gamma}$ induced by the strategy pair $(\pi, \gamma)$.
The problem of constructing a distributionally robust safe policy can then be formulated as
the following zero-sum  dynamic game problem:
\begin{equation} \label{reach}
\sup_{\pi \in \Pi} \inf_{\gamma \in \Gamma} P_{\bm{x}}^{\mbox{\tiny safe}} (\pi, \gamma; A).
\end{equation}
In this two-player  game, Player I determines a control policy $\pi$ to maximize the probability of safety assuming that Player II selects a disturbance distribution strategy $\gamma$ to minimize the  probability of safety.
Recall that information about the ambiguity set $\mathbb{D}_t$ of probability distributions is encoded in Player II's strategy space $\Gamma$.
In general, the action space of Player II at each stage is infinite-dimensional because the disturbance may have a continuous probability density.
Therefore, the Bellman equation associated with this dynamic game problem involves infinite-dimensional optimization problems, which are computationally challenging.  
To alleviate this computational issue, we  propose a duality-based approach to reformulate the Bellman equation as a semi-infinite program that can be solved by convergent algorithms.
In the next section, we first establish some analytical results about the dynamic game problem. 
In particular, we show that under mild conditions an associated value function is upper semi-continuous and 
a non-randomized Markov control policy is optimal.

\section{Dynamic Programming}\label{dp_sol}

\subsection{A Semi-Continuous Model}

We let $\mathbb{K}_t \in \mathcal{B}(\mathbb{R}^n \times \mathbb{R}^m \times M_+(\mathbb{R}^l))$ be the collection of elements such that each $(\bm{x}, \bm{u}, \bm{\mu}) \in \mathbb{K}_t$ satisfies $(i)$ $\bm{u} \in \mathbb{U}(\bm{x})$ and $(ii)$ $\bm{\mu} \in \mathbb{D}_t$. 
The stochastic kernel $Q_t(\bm{\xi} | \bm{x},  \bm{u}, \bm{\mu})$
represents the probability that the state of the system \eqref{sys} at stage $t+1$ is equal to $\bm{\xi} \in \mathbb{R}^n$ given $(x_t,u_t, \mu_t) = (\bm{x}, \bm{u}, \bm{\mu})$.

\begin{assumption}\label{sc}
The problem \eqref{reach} of constructing a distributionally robust safe policy satisfies the following conditions:
\begin{enumerate}[$(i)$]

\item For each bounded continuous function $g_t :\mathbb{R}^n  \to \mathbb{R}$, the function 
\[
\hat{g}_t(\bm{x},  \bm{u}, \bm{\mu}) := \int_{\mathbb{R}^n} g_t(\xi) Q_t(d\xi | \bm{x},  \bm{u}, \bm{\mu})
\]
is continuous on $\mathbb{K}_t$ for each $t \in \mathcal{T}$.

\item The set $\mathbb{U}(\bm{x})$ is compact for each $\bm{x} \in \mathbb{R}^n$. Furthermore, the set-valued mapping $\bm{x} \mapsto \mathbb{U}(\bm{x})$ is upper semi-continuous.

\item The ambiguity set $\mathbb{D}_t$ is $\sigma$-compact for each $t \in \mathcal{T}$.
\end{enumerate}
\end{assumption}

These conditions are standard \emph{measurable selection conditions} for semi-continuous stochastic control models (e.g., \cite{Dubins1965}, \cite{Hernandez2012}, \cite{Gonzalez2003}) and will be used to ensure the existence of a distributionally robust safe policy.
We will further show that a non-randomized Markov policy is optimal.
For each measurable function $\bold{v}$ on $\mathbb{R}^n$, we define a dynamic programming operator, denoted by $\bold{T}_t$, $t \in \mathcal{T}$, as
\[
\bold{T}_t \bold{v}(\bm{x}, \bm{y}) := \sup_{\bm{u} \in \mathbb{U}(\bm{x})} \inf_{\bm{\mu} \in \mathbb{D}_t}  \bold{1}_A(\bm{x})
 \int_{\mathbb{R}^n}  \bold{v}( \xi) Q_t (d \xi | \bm{x}, \bm{u}, \bm{\mu}).
\]
We can first show the following properties of the dynamic programming operator. In particular, the outer ``$\sup$'' problem has an optimal solution under a mild condition on $\bold{v}$.

\begin{lemma} \label{sc_lem}
Let $\bold{v}:\mathbb{R}^n \to \mathbb{R}$ be a measurable upper semi-continuous function with $\| \bold{v} \|_\infty < \infty$ and $\bold{v} \geq 0$. Then, 
\begin{enumerate}[$(i)$]
\item  $\bold{T}_t \bold{v}$ is upper semi-continuous.

\item  There exists a measurable function $\kappa: \mathbb{R}^n \to \mathbb{R}^m$ such that, for all $\bm{x} \in \mathbb{R}^n$, $\kappa(\bm{x}) \in \mathbb{U}(\bm{x})$ and
\[
\bold{T}_t\bold{v}(\bm{x}) = \inf_{\bm{\mu} \in \mathbb{D}_t} 
\left [
\bold{1}_A(\bm{x}) \int_{\mathbb{R}^n} \bold{v}(\xi) Q_t(d\xi | \bm{x}, \kappa(\bm{x}), \bm{\mu})
\right ].
\]
\end{enumerate}
\end{lemma}
\begin{proof}
Define a function $\hat{\bold{v}}: \mathbb{R}^n \times \mathbb{R}^m \times M_+(\mathbb{R}^l) \to \mathbb{R}$ as
\[
\hat{\bold{v}}(\bm{x}, \bm{u}, \bm{\mu}) :=
\int_{\mathbb{R}^n} \bold{v}(\xi) Q_t(d\xi | \bm{x}, \bm{u}, \bm{\mu}).
\]
Since $\bold{v}$ is a measurable, upper semi-continuous and nonnegative function with finite $L^\infty$-norm, there exists a sequence $\{ \bold{v}_k\}$ such that $\bold{v}_k \downarrow \bold{v}$ pointwise and each $\bold{v}_k$ is a bounded continuous function.
Thus, for all $k$, $\int_{\mathbb{R}^n} \bold{v}(\xi) Q_t(d\xi | \bm{x}, \bm{u}, \bm{\mu}) \leq \int_{\mathbb{R}^n} \bold{v}_k(\xi) Q_t(d\xi | \bm{x}, \bm{u}, \bm{\mu})$, and for any $(\bm{x}_j, \bm{u}_j, \bm{\mu}_j) \to (\bm{x}, \bm{u}, \bm{\mu})$ we have that for all $k$
\begin{equation} \nonumber
\begin{split}
&\lim \sup_{j \to \infty} \int_{\mathbb{R}^n} \bold{v}(\xi) Q_t(d\xi | \bm{x}_j, \bm{u}_j, \bm{\mu}_j) \leq \int_{\mathbb{R}^n} \bold{v}_k(\xi) Q_t(d\xi | \bm{x}, \bm{u}, \bm{\mu}) 
\end{split}
\end{equation}
due to Assumption \ref{sc} $(ii)$.
Letting $k \to \infty$, we  conclude that $\hat{\bold{v}}$ is   upper semi-continuous. 
Since $\bold{1}_A: \mathbb{R}^n \to \mathbb{R}$ is upper semi-continuous with a compact set $A$, 
$(\bm{x}, \bm{u}, \bm{\mu}) \mapsto \bold{1}_A(\bm{x}) \hat{\bold{v}}(\bm{x}, \bm{u}, \bm{\mu})$
is upper semi-continuous as well.
Furthermore, for all $(\bm{x}, \bm{u}, \bm{\mu}) \in \mathbb{K}_t$,
\begin{equation} \nonumber
\begin{split}
|\bold{1}_A(\bm{x}) \hat{\bold{v}}(\bm{x}, \bm{u}, \bm{\mu})| &= 
\bold{1}_A(\bm{x}) \int_{\mathbb{R}^n} |\bold{v}(\xi)| Q_t(d\xi | \bm{x}, \bm{u}, \bm{\mu})\\
& \leq \| \bold{v} \|_\infty < \infty.
\end{split}
\end{equation}
Thus, by Lemma 3.2. (b) in \cite{Gonzalez2003} have that
$\bold{T}_t \bold{v}$ is upper semi-continuous and that there exists a measurable function $h:\mathbb{R}^n \to \mathbb{R}^m$ such that 
$\bold{T}_t\bold{v}(\bm{x}) = \inf_{\bm{\mu} \in \mathbb{D}_t} 
\left [
\bold{1}_A(\bm{x}) \int_{\mathbb{R}^n} \bold{v}(\xi) Q_t(d\xi | \bm{x}, \kappa(\bm{x}), \bm{\mu})
\right ]$ and $\kappa(\bm{x}) \in \mathbb{U}(\bm{x})$ for all $\bm{x} \in \mathbb{R}^n$ under Assumption~\ref{sc}~$(iii)$ and $(iv)$.
\end{proof}

For each $t \in \mathcal{T}$, we define the value function of the distributionally robust safe control problem \eqref{reach} as
\[
v_t(\bm{x}) := \bold{T}_t \circ \bold{T}_{t+1} \circ \cdots \circ \bold{T}_{T-1} \bold{1}_A(\bm{x}).
\]
It represents the maximal worst-case probability of the system being safe from stage $t$ to $T$ when $x_t = \bm{x}$. For $t = T$, the value function is defined as $v_T(\bm{x}) = \bold{1}_A(\bm{x})$.
By definition, $v_0 (x_0) = \sup_{\pi \in \Pi} \inf_{\gamma \in \Gamma} P_{x_0}^{\mbox{\tiny safe}} (\pi, \gamma; A)$ for some initial state $x_0$.
Setting $\bold{v} = v_{t+1}$ in Lemma \ref{sc_lem}, 
we can show that the distributionally robust safe control problem \eqref{reach} admits a non-randomized Markov policy, which is optimal. 

\begin{theorem}[A Markov policy is optimal]\label{exist}
Suppose that Assumption \ref{sc} holds. For each $t \in \mathcal{T}$, there exists a measurable function $\phi_t : \mathbb{R}^n \to \mathbb{R}^m$ such that $\phi_t(\bm{x}) \in \mathbb{U}(\bm{x})$ and
\[
v_t(\bm{x}) = \inf_{\bm{\mu} \in \mathbb{D}_t} \left [ \bold{1}_A(\bm{x})
\int_{\mathbb{R}^n} v_{t+1} (\xi) Q_t(d\xi | \bm{x}, \phi_t(\bm{x}), \bm{\mu}) 
\right ]
\]
for all $\bm{x} \in \mathbb{R}^n$. The non-randomized Markov policy $\pi^\star := (\phi_0, \cdots, \phi_{T-1}) \in \Pi$ is a distributionally robust safe policy, i.e.,
\[
v_0(\bm{x}) = \inf_{\gamma \in \Gamma} P_{\bm{x}}^{\emph{\mbox{\tiny safe}}} (\pi^\star, \gamma; A). 
\]
Furthermore, the value function $v_t$ is upper semi-continuous for each $t \in \bar{\mathcal{T}}$.
\end{theorem}
\begin{proof}
We first show that $v_t$ is a measurable upper semi-continuous function with $\| v_t \|_\infty < \infty$ for each $t$ by mathematical induction.
For $t = T$, $v_T = \bold{1}_A$, which is upper semi-continuous because $A$ is closed. Furthermore, it is measurable and has a finite $L^\infty$-norm. Suppose now that $v_{t+1}$ is a measurable upper semi-continuous function with $\| v_{t+1} \|_\infty < \infty$
Then, by Lemma \ref{sc_lem} $(i)$, $v_t = \bold{T}_t v_{t+1}$ is upper semi-continuous. Furthermore, it is clear that $v_t$ is measurable and $\| v_t \|_\infty \leq \| v_{t+1} \|_\infty < \infty$. This completes our inductive argument. 

We now use Lemma \ref{sc_lem} $(ii)$ to conclude that there exists a measurable function $\phi_t :\mathbb{R}^n \to \mathbb{R}^m$ such that 
$v_t(\bm{x}) = \sup_{\bm{\mu} \in \mathbb{D}_t} [ \bold{1}_A(\bm{x})
\int_{\mathbb{R}^n} v_{t+1} (\xi) Q_t(d\xi | \bm{x}, \phi_t(\bm{x}), \bm{\mu}) ]$ and $\phi_t(\bm{x}) \in \mathbb{U}(\bm{x})$ for all $(t, \bm{x}) \in \mathcal{T} \times \mathbb{R}^n$. Lastly, by the dynamic programming principle, we obtain  that $v_0(\bm{x}) = \inf_{\gamma \in \Gamma} \mathbb{E}^{\pi^\star, \gamma}  [
\prod_{t=0}^T \bold{1}_A(x_t) ]$ with $x_0 = \bm{x}$.
\end{proof}

This theorem greatly reduces the control strategy space we need to search for because it suffices to restrict our attention to non-randomized Markov policies. 
The Markov policy $\pi^\star$ maximizes the worst-case  probability of safety no matter how the disturbance (Player II) chooses its probability distribution $\mu_t$ in the ambiguity set $\mathbb{D}_t$  using the history of information. 
Furthermore, the dynamic programming principle allows us to obtain the following Bellman equation:

\begin{proposition}
Suppose that Assumption \ref{sc} holds. The value function $v_t$ solves the following Bellman equation:
for $t = T$
\begin{equation} \nonumber
v_T(\bm{x}) = \bold{1}_A(\bm{x}),
\end{equation}
and for $t \in \mathcal{T}$
\begin{equation}\label{dp}
\begin{split}
&v_t(\bm{x}) = \max_{\bm{u} \in \mathbb{U}(\bm{x})}
\inf_{\bm{\mu} \in \mathbb{D}_t}  \bold{1}_A(\bm{x}) \int_{\mathbb{R}^l} v_{t+1} (f(\bm{x}, \bm{u}, \bm{w})) d\bm{\mu} (\bm{w}).
\end{split}
\end{equation}
\end{proposition}
Note that ``$\max$'' is used instead of ``$\sup$'' in the outer problem because its optimal solution exists due to Theorem \ref{exist}.
This Bellman equation is computationally challenging because
it involves infinite-dimensional minimax optimization problems over the ambiguity set $\mathbb{D}_t$ of probability distributions.
In Section \ref{dual_Bellman}, we propose a duality-based approach to alleviate the computational issue that arises from the infinite-dimensional minimax programs.
Before introducing this method,
we discuss how to construct distributionally robust safe sets and controllers through a stochastic reachability analysis in the following subsection.

\subsection{Constructing Distributionally Robust Safe Sets and Safety-Oriented Controllers}\label{socon}

To obtain distributionally robust control policies and safe sets,
we evaluate the value function  $\{v_t\}_{t=0}^{T}$
by solving the Bellman equation backward in time, i.e., from $t = T$ to $t = 0$. 
A distributionally robust safe policy can then be constructed as
\begin{equation} \nonumber
\begin{split}
&\phi_t^{\tiny \mbox{safe}} (x_t) \in  \arg\max_{\bm{u} \in \mathbb{U}(x_t)} \inf_{\bm{\mu} \in \mathbb{D}_t}  \bold{1}_A(x_t) \int_{\mathbb{R}^l} v_{t+1} ( f(x_t, \bm{u}, \bm{w}) )  d\bm{\mu}(\bm{w}).
\end{split}
\end{equation}
The distributionally robust safe set with probability $\alpha$ can be computed as
\[
S_\alpha^\star(A) = \{\bm{x} \in \mathbb{R}^n \: | \:
v_0(\bm{x}) \geq \alpha \}
\]
because $v_0(\bm{x}) = \max_{\pi \in \Pi} \inf_{\gamma \in \Gamma} P_{\bm{x}}^{\mbox{\tiny safe}}(\pi, \gamma; A)$.  

Define the $t$-distributionally robust safe set with  probability $\alpha$  as
\begin{equation} \nonumber
\begin{split}
S_{\alpha, t}^\star (A) := \{
&\bm{x} \in \mathbb{R}^n \: | \: 
\forall \gamma \in \Gamma \;
\exists \pi \in\Pi \mbox{ s.t. } \mathbb{P}^{\pi, \gamma} (x_s \in A \; \forall s = t, \cdots, T, x_t = \bm{x} ) \geq \alpha
\}.
\end{split}
\end{equation}
If the system state lies in $S_{\alpha, t}^\star (A)$ at stage $t$, 
there exists a control policy such that the probability of the system being safe during the remaining stages is greater than equal to~$\alpha$.
Note that $S_{\alpha, 0}^\star(A) = S_{\alpha}^\star(A)$ under Assumption \ref{sc} which guarantees the existence of a distributionally robust safe policy.
Using the value function, we can  compute the set as
\[
S_{\alpha, t}^\star (A) = \{ \bm{x} \in \mathbb{R}^n \: | \: v_t (\bm{x}) \geq \alpha \}
\]
because  $v_t(\bm{x}) = \max_{\pi \in \Pi} \inf_{\gamma \in \Gamma} \mathbb{E}^{\pi, \gamma} [ \prod_{s = t}^T \bold{1}_A(x_s) ]$ with $x_t = \bm{x}$.  
Suppose now that we are given the support $W_t$ of $\mu_t$ for each $t$.
Consider the following controller:
given $x_t$ at stage $t$, the control action is determined as
\begin{equation} \label{soc}
u_t^{\tiny \mbox{safe}} \left \{ 
\begin{array}{ll}
 \in \mathbb{U}(x_t) 
& \mbox{ if }x_t \in \{ \bm{x} \: | \: f(\bm{x}, \bm{u}, \bm{w}) \in S_{\alpha, t+1}^\star(A) \; \forall \bm{u} \in \mathbb{U}(\bm{x})\; \forall \bm{w} \in W_t  \}\\
= \phi_t^{\tiny \mbox{safe}} (x_t)   & \mbox{ otherwise}.
\end{array}
\right.
\end{equation}
This controller chooses an arbitrary admissible control action if it can drive the system into $S^\star_{\alpha, t+1} (A)$ at stage $t+1$ for any realization of the disturbance. Otherwise, it uses a distributionally robust safe policy. 
This procedure is motivated by the deterministic safe controller synthesis method of
 Lygeros et al. \cite{Lygeros1999}. 
 We can show that  at each $t$ this controller ensures that 
 the system will remain safe for stages $t+1, \cdots, T$
 with probability greater than or equal to $\alpha$ 
under any disturbance distribution strategy $\gamma \in \Gamma$. 
 \begin{proposition}
 Suppose that Assumption \ref{sc} holds. 
 If the initial state is chosen so that $x_0 \in S_\alpha^\star (A)$ and 
 the Markov control policy \eqref{soc} is used,  then
 \[
x_t \in S_{\alpha,t}^\star(A) \quad \forall t \in \bar{\mathcal{T}},
\]
i.e., the probability for the system being safe for all remaining stages is greater than or equal to $\alpha$, regardless of how the disturbance distribution is chosen in the ambiguity set.
 \end{proposition}
 
 \begin{proof}
 We use mathematical induction.
 For stage $t = 0$, the statement is true since $x_0$ is assumed to be contained in $S_\alpha^\star (A) = S_{\alpha, 0}^\star (A)$. 
Suppose that $x_t \in S_{\alpha, t}^\star (A)$ for some $t \in \mathcal{T}$.
 Fix an arbitrary $t \in \mathcal{T}$.
Assume first that $x_t \in \{ \bm{x} \: | \: f(\bm{x}, \bm{u}, \bm{w}) \in S_{\alpha, t+1}^\star(A)
\; \forall \bm{u} \in \mathbb{U}(\bm{x})\; \forall \bm{w} \in W_t  \}$. 
Fix an arbitrary $u_t \in \mathbb{U}(x_t)$.
The controller guarantees that
$x_{t+1} = f(x_t, u_t, w_t) \in S_{\alpha, t+1}^\star (A)$  with probability 1 for any $\mu_t \in \mathbb{D}_t$
because ${\mu}_t(W_t) = \int_{W_t} d{\mu}_t(\bm{w}) = 1$.
On the other hand, when  $x_t \notin \{ \bm{x} \: | \: f(\bm{x}, \bm{u}, \bm{w}) \in S_{\alpha, t+1}^\star(A)
\; \forall \bm{u} \in \mathbb{U}(\bm{x})\; \forall \bm{w} \in W_t  \}$,
by using a distributionally robust safe policy $\phi_t^{\mbox{\tiny safe}}$, we can ensure that $x_{t+1} \in S_{\alpha, t+1}^\star (A)$ since
\begin{equation}\nonumber
\begin{split}
\mathbb{P}^{\pi^{\mbox{\tiny safe}}, \gamma}(x_s \in A \; \forall s \in \mathcal{T}_{t+1})
&\geq \mathbb{P}^{\pi^{\mbox{\tiny safe}}, \gamma}(x_s \in A \; \forall s \in \mathcal{T}_{t})\\
& \geq \alpha \quad \forall \gamma \in \Gamma,
\end{split}
\end{equation}
where $\pi^{\mbox{\tiny safe}} := \{  \phi_0^{\tiny \mbox{safe}}, \cdots, \phi_{T-1}^{\tiny \mbox{safe}}  \}$ and 
$\mathcal{T}_{t} := \{t, t+1, \cdots, T\}$.
 This completes our inductive argument.
 \end{proof}

 If another objective function needs to be minimized in a distributionally robust way while ensuring safety, one may solve
\[
\inf_{\pi \in \Pi} \sup_{\gamma \in \Gamma} \mathbb{E}^{\pi, \gamma} \left [ \sum_{t=0}^{T-1} r(x_t, u_t) + q(x_T) \right ]
\]
to obtain an optimal distributionally robust policy $\pi^{opt}$, where $r :\mathbb{R}^n \times \mathbb{R}^m \to \mathbb{R}$ is a running cost function and $q:\mathbb{R}^n \to \mathbb{R}$ is a terminal cost function of interest. 
Then, one can employ the proposed controller \eqref{soc} that chooses $\pi_t^{opt} (x_t)$ whenever
$x_t \in \{ \bm{x} \: | \: f(\bm{x}, \bm{u}, \bm{w}) \in S_{\alpha, t+1}^\star(A)\;\forall \bm{u} \in \mathbb{U}(\bm{x})\; \forall \bm{w} \in W_t\}$.
Note that this controller prioritizes safety and tries to minimize the worst-case cost value whenever there is the flexibility to do so. 
This \emph{safety-oriented distributionally robust control design} approach can be overly conservative, particularly when $W_t$ is large.
However, it is  computationally efficient because the cost-minimizing control problem is decoupled from the safe control problem.

\section{Moment Uncertainty and Dual Bellman Equations}\label{dual_Bellman}

\subsection{Ambiguity Sets with Moment Uncertainty}\label{amb_moment}

Recall that the proposed distributionally robust safe policies maximize the worst-case probability for a system to be safe, assuming that the probability distribution of the disturbance lies within an ambiguity set, $\mathbb{D}_t$, of distributions.
Therefore, modeling the ambiguity set may critically affect the resulting safe policies. 
Several ambiguity set modeling approaches have been developed in the context of single-stage optimization problems.
The approaches can be categorized as \emph{moment-based} and \emph{statistical distance-based} methods.
A moment-based approach employs an ambiguity set of distributions whose moments (e.g., mean and covariance) satisfy certain constraints \cite{Scarf1958}, \cite{Delage2010}, \cite{Popescu2007}, \cite{Zymler2013}, \cite{Wiesemann2014}.
A statistical distance-based approach takes into account an ambiguity set of probability distributions that are closed to a nominal distribution in terms of a chosen statistical distance, such as $\phi$-divergence \cite{Bayraksan2015}, \cite{BenTal2013}, \cite{Jiang2016}, \cite{Sun2016}, Prokhorov metric \cite{Erdogan2006} and Wasserstein distance \cite{Esfahani2015}, \cite{Zhao2015}, \cite{Gao2016}.

In this work, we take a moment-based approach.
Suppose that an estimate of the mean and covariance matrix of the disturbance $w_t$ is the only available information.
Let $\bold{m}_t \in \mathbb{R}^l$ and $\bold{\Sigma}_t \in \mathbb{R}^{l \times l}$ be the estimate of the mean and covariance matrix, respectively.
The set of all the probability distributions (i.e., distribution measures) that are consistent with these estimates can be modeled as
\begin{equation} \label{a_set}
\begin{split}
\mathbb{D}_t := \{ &\mu_t \in M_+(W_t) \: | \: \mu_t(W_t) = 1, \\
&|\mathbb{E}_{\mu_t} [ w_t] - \bold{m}_t |  \leq b_t, \\
&\mathbb{E}_{\mu_t} [ (w_t - \bold{m}_t)(w_t - \bold{m}_t)^\top ] \preceq c_t \bold{\Sigma}_t\},
\end{split}
\end{equation}
where $\mathbb{E}_{\mu_t}$ denotes the expectation taken with respect to the probability distribution $\mu_t$. 
Here, $b_t \in \mathbb{R}_+^l$ and $c_t \geq 1$ are given constants that depend on one's confidence in the estimates $\bold{m}_t$ and $\bold{\Sigma}_t$.
Any probability distribution in this set satisfies the following properties: $(i)$ the support of $w_t$ is $W_t$; $(ii)$ the mean of $w_{t, i}$ lies in a circle of size $b_{t,i}$; and $(iii)$ the centered second moment matrix of $w_t$ lies in a positive semidefinite cone. It models how likely $w_t$ is to be close to $\bold{m}_t$ in terms of the correlation matrix $c_t \bold{\Sigma}_t$ with $c_t \geq 1$.

\subsection{Zero Duality Gap}

In general, solving the Bellman equation~\eqref{dp} to evaluate $v_t (\bm{x})$ with the ambiguity set \eqref{a_set}
 is challenging as it involves infinite-dimensional minimax optimization problems.
To resolve this difficulty, we propose a dual formulation method.\footnote{The proposed method does not resolve the scalability issue inherent in dynamic programming; the complexity of computing $v_t(\bm{x})$ is still exponential with the dimension of system state $\bm{x}$ even if the proposed approach is employed.}
Fix an arbitrary $(t, \bm{x}, \bm{u}) \in \mathcal{T} \times \mathbb{R}^n \times \mathbb{R}^m$ such that $\bm{u} \in \mathbb{U}(\bm{x})$.
With the ambiguity set \eqref{a_set},
the inner minimization problem in
the Bellman equation \eqref{dp} can be written as the following infinite-dimensional conic linear program:
\begin{subequations} \label{primal}
\begin{align}
\mathbf{P}: \inf_{\bm{\mu}} \;\; & \int_{W_t} v_{t+1} (f(\bm{x}, \bm{u}, \bm{w})) d\bm{\mu}(\bm{w}) \\
\mbox{s.t.} \;\; &  \bold{m}_t - b_t \leq \int_{W_t} \bm{w} d\bm{\mu}(\bm{w}) \leq \bold{m}_t + b_t  \label{con1} \\
& \int _{W_t} (\bm{w} - \bold{m}_t) (\bm{w} - \bold{m}_t)^\top d\bm{\mu} (\bm{w}) \preceq c_t \bold{\Sigma}_t \label{con2} \\
& \int_{W_t} d \bm{\mu}(\bm{w}) = 1 \label{con3} \\
&\bm{\mu} \in M_+(W_t). \label{con4}
\end{align}
\end{subequations}
Here, \eqref{con1} and \eqref{con2} represent the first- and second-moment constraints encoded in the ambiguity set $\mathbb{D}_t$, respectively.
The constraints \eqref{con3} and \eqref{con4} ensure that $\bm{\mu}$ is a probability distribution measure whose support is $\mathbb{D}_t$.
Let $\underline{\bm{\lambda}}, \overline{\bm{\lambda}} \in \mathbb{R}^l$ be the Lagrange multipliers associated with the inequality constraints \eqref{con1}.
We also let $\bm{\Lambda} \in \mathbb{S}^l$ and $\bm{\nu} \in \mathbb{R}$ be the Lagrange multipliers associated with \eqref{con2} and \eqref{con3}, respectively.
Its dual problem can then be derived as
\begin{equation} \label{dual}
\begin{split}
\mathbf{P}^*: \sup_{\underline{\bm{\lambda}}, \overline{\bm{\lambda}}, \bm{\Lambda}, \bm{\nu}} \;  &-\underline{b}_t^\top \underline{\bm{\lambda}}  - \overline{b}_t^\top \overline{\bm{\lambda}} - c_t \mbox{Tr}(\bold{\Sigma}_t \bm{\Lambda}) - \bm{\nu}\\
\mbox{s.t.} \;& \bm{w}^\top (\overline{\bm{\lambda}} - \underline{\bm{\lambda}}) + (\bm{w} - \bold{m}_t)^\top \bm{\Lambda} (\bm{w} - \bold{m}_t) \\
& + \bm{\nu} + v_{t+1}(f(\bm{x}, \bm{u}, \bm{w})) \geq 0 \; \forall \bm{w} \in W_t\\
& \underline{\bm{\lambda}}, \overline{\bm{\lambda}} \geq 0, \; \bm{\Lambda} \succeq 0\\
& \underline{\bm{\lambda}}, \overline{\bm{\lambda}} \in \mathbb{R}^l, \; \bm{\Lambda}  \in \mathbb{S}^l, \; \bm{\nu} \in \mathbb{R},
\end{split}
\end{equation}
where
$\underline{b}_t := b_t - \bold{m}_t$ and $\overline{b}_t := b_t + \bold{m}_t$.
This is a semi-infinite program because the first constraint must be satisfied for all $\bm{w} \in W_t \subseteq \mathbb{R}^l$.
Let $\inf \mathbf{P}$ and $\sup \mathbf{P}^*$ denote the optimal values of the primal and dual problems, respectively. 
By weak duality, we have
$\sup \mathbf{P}^* \leq \inf \mathbf{P}$.
However, we can further show that strong duality holds, i.e., the dual program is exact in the sense that the duality gap is zero.

\begin{proposition}[Zero duality gap]
Suppose that Assumption \ref{sc} holds, $\mathbf{P}$ has a feasible solution, and $W_t$ is compact.
Then,  $\mathbf{P}$ has an optimal solution and there is no duality gap, i.e.,
\[
\sup \mathbf{P}^* = \min \mathbf{P}.
\]
\end{proposition}

\begin{proof}
Note that $\mathbf{P}$ has a feasible solution with finite value since the objective function value lies in $[0,1]$.
We introduce the following convex cone:
\begin{equation}\nonumber
\begin{split}
P(W_t) := \big \{
&(\underline{\bm{\lambda}}, \overline{\bm{\lambda}}, \Lambda, \bm{\nu}, {\lambda}_0 )\in \mathbb{R}^l \times \mathbb{R}^l \times \mathbb{S}^l \times \mathbb{R} \times \mathbb{R} : \underline{\bm{\lambda}}, \overline{\bm{\lambda}} \geq 0, \bm{\Lambda} \succeq 0; \\
& \bm{w}^\top (\overline{\bm{\lambda}} - \underline{\bm{\lambda}})
+(\bm{w} - \bold{m}_t)^\top \bm{\Lambda} (\bm{w} - \bold{m}_t)
+\bm{\nu} + \lambda_0 v_{t+1}(f(\bm{x}, \bm{u}, \bm{w})) \geq 0 \; \forall \bm{w} \in W_t
\big \}.
\end{split}
\end{equation}
Fix an arbitrary $\epsilon > 0$. Choose $(\underline{\bm{\lambda}}^{\mbox{\tiny feas}}, \overline{\bm{\lambda}}^{\mbox{\tiny feas}}, \Lambda^{\mbox{\tiny feas}}, \bm{\nu}^{\mbox{\tiny feas}}, 1)$ $\in P(W_t)$ such that 
$\underline{\bm{\lambda}}^{\mbox{\tiny feas}} = \overline{\bm{\lambda}}^{\mbox{\tiny feas}} = 0$, $\bm{\Lambda}^{\mbox{\tiny feas}} = \epsilon I$ and 
\[
\bm{\nu}^{\mbox{\tiny feas}} =\rho - \inf_{\bm{w} \in W_t, \delta \in [-\epsilon, \epsilon]} (1 + \delta) v_{t+1}(f(\bm{x}, \bm{u}, \bm{w})),
\]
where
\[
\rho := \epsilon - \inf_{\bm{w} \in W_t, \delta' \in [-2\epsilon, 2\epsilon]^l}  \bm{w}^\top \delta'.
\]
Note that $\rho \in \mathbb{R}$ because $W_t$ is compact, and that $\bm{\nu}^{\mbox{\tiny feas}}  \in \mathbb{R}$ because the value of $v_{t+1}$ is in $[0, 1]$.
Therefore, any $(\underline{\bm{\lambda}}, \overline{\bm{\lambda}}, \Lambda, \bm{\nu}, {\lambda}_0 )$ that is contained in the $\epsilon$-ball centered at $(\underline{\bm{\lambda}}^{\mbox{\tiny feas}}, \overline{\bm{\lambda}}^{\mbox{\tiny feas}}, \Lambda^{\mbox{\tiny feas}}, \bm{\nu}^{\mbox{\tiny feas}}, 1 )$ lies in the cone $P(W_t)$.
This implies that $(\underline{\bm{\lambda}}^{\mbox{\tiny feas}}, \overline{\bm{\lambda}}^{\mbox{\tiny feas}}, \Lambda^{\mbox{\tiny feas}}, \bm{\nu}^{\mbox{\tiny feas}}, 1 )$ is an interior point of  $P(W_t)$.
Thus, due to Theorem 1.2 in \cite{Lasserre2009}, there is no duality gap and the inf in $\mathbf{P}$ is attained.

\end{proof}
This proof ensures that a Slater type condition holds, and the rest follows from results of conic duality in infinite-dimensional convex optimization (see also \cite{Shapiro2001}).

\subsection{Dual Bellman Equation as a Semi-Infinite Program}

Using the zero duality gap result, we can substitute the inner minimization problem ($\inf \mathbf{P}$) in the Bellman equation as its dual problem ($\sup \mathbf{P}^*$) without sacrificing optimality.
This substitution allows us to evaluate the value function $v_t(\bm{x})$ via a dual version of the Bellman equation.
\begin{theorem}[Dual Bellman equation] \label{reform}
Suppose that Assumption \ref{sc} holds, $\mathbf{P}$ has a feasible solution, and $W_t$ is compact.
For all $(t, \bm{x}) \in \mathcal{T} \times  \mathbb{R}^n$, the Bellman equation \eqref{dp} is equivalent to the following semi-infinite program:
\begin{equation} \nonumber
\begin{split}
v_t(\bm{x}) =\: \bold{1}_A(\bm{x}) \times
\sup_{\bm{u}, \underline{\bm{\lambda}}, \overline{\bm{\lambda}}, \bm{\Lambda}, \bm{\nu}} \; & -\underline{b}_t^\top \underline{\bm{\lambda}}- \overline{b}_t^\top \overline{\bm{\lambda}}  - c_t  \mbox{\emph{Tr}}(\bold{\Sigma}_t \bm{\Lambda}) - \bm{\nu}\\
\mbox{s.t.} \; &  \bm{w}^\top (\overline{\bm{\lambda}} - \underline{\bm{\lambda}}) + (\bm{w} - \bold{m}_t)^\top \bm{\Lambda} (\bm{w} - \bold{m}_t) + \bm{\nu} \\
& + v_{t+1}(f(\bm{x}, \bm{u}, \bm{w})) \geq 0 \quad \forall \bm{w} \in W_t\\
& \underline{\bm{\lambda}}, \overline{\bm{\lambda}} \geq 0, \; \bm{\Lambda} \succeq 0\\
& \bm{u} \in \mathbb{U}(\bm{x}), \; \underline{\bm{\lambda}}, \overline{\bm{\lambda}} \in \mathbb{R}^l, \; \bm{\Lambda}  \in \mathbb{S}^l, \; \bm{\nu} \in \mathbb{R}
\end{split}
\end{equation}
with the terminal condition $v_T(\bm{x}) = \bold{1}_A(\bm{x})$.
\end{theorem}

\begin{remark}
For a compact representation, we merged \emph{``$\max_{\bm{u}}$''} and \emph{``$\sup_{\underline{\bm{\lambda}},\overline{\bm{\lambda}}, \bm{\Lambda}, \bm{\nu}}$''}. However, it should be noted that this semi-infinite program has an optimal feasible $\bm{u}$.
\end{remark}

Since the dual problem $(\sup \mathbf{P}^*)$ is a  semi-infinite program, the dual Bellman equation also involves semi-infinite optimization problems.
Each semi-infinite optimization program can be solved by using several convergent methods, such as discretization methods, exchange methods, homotopy methods, and primal-dual methods (e.g., \cite{Hettich1993}, \cite{Reemtsen1998}, \cite{Lopez2007} and the reference therein).
In Section \ref{ex}, we employ the discretization method proposed by Reemtsen~\cite{Reemtsen1991}.
This algorithm adaptively generates a grid on $W_t$ and converges to  a locally optimal value of the semi-infinite program in the dual Bellman equation (see \cite{Reemtsen1991} for a proof).
When the semi-infinite program is concave, it converges to the globally optimal value.
We can show that each semi-infinite program is  concave under the next assumption.

\begin{assumption} \label{struct}
The following  properties  hold:
\begin{enumerate}[$(i)$]
\item $f: \mathbb{R}^n \times \mathbb{R}^m \times \mathbb{R}^l \to \mathbb{R}^n$ is an affine function;

\item  $A$ is a convex set;

\item For all $\bm{x}_1, \bm{x}_2 \in \mathbb{R}^n$ and for all $\lambda \in (0,1)$, 
if $\bm{u}_i \in \mathbb{U}(\bm{x}_i)$, $i=1,2$, then
$\lambda \bm{u}_1 + (1-\lambda) \bm{u}_2 \in \mathbb{U}(\lambda \bm{x}_1 + (1-\lambda) \bm{x}_2)$.
\end{enumerate}

\end{assumption}

When $\mathbb{U}$ is independent of $\bm{x}$,  Assumption \ref{struct} $(iii)$ is equivalent to the concavity of $\mathbb{U}$.
Note that this assumption is independent of the ambiguity set $\mathbb{D}_t$.  Thus, the following concavity result holds with general ambiguity sets:

\begin{proposition}\label{concave}
Suppose  that Assumptions \ref{sc} and \ref{struct} hold.
Then, the value function $v_t (\bm{x})$ is concave with respect to $\bm{x} \in A$ for each $t \in \mathcal{T}$.
\end{proposition}

Its proof is contained in the Appendix.  
Since $\bm{u} \mapsto v_t(f(\bm{x}, \bm{u}, \bm{w}))$ is concave for each $(t, \bm{x}, \bm{w}) \in \bar{\mathcal{T}} \times A \times W_t$ under Assumption \ref{struct}, the semi-infinite program in the dual Bellman equation is concave for each $(t, \bm{x}) \in \mathcal{T} \times A$.

\begin{remark}
When control actions are chosen from a discrete set, i.e., $\mathbb{U}(\bm{x})$ is a discrete set, the semi-infinite program in Theorem \ref{reform} can be formulated as a mixed-integer program, where  \emph{``$\max_{\bm{u}}$''} is a discrete-optimization problem and \emph{``$\sup_{\underline{\bm{\lambda}},\overline{\bm{\lambda}}, \bm{\Lambda}, \bm{\nu}}$''} is a continuous optimization problem.
In particular, under Assumption \ref{struct}, we can employ a linear approximation-based method to obtain an approximate solution with a provable suboptimality bound~\cite{Yang2016}.
\end{remark}

\section{Numerical Examples: Thermostatically Controlled Loads}\label{ex}

Thermostatically controlled loads (TCLs)---such as air conditioners, refrigerators, water heaters, and battery pack cooling systems---are used to guarantee
human comfort, and food provision and battery safety, etc.
Therefore, ensuring safe  TCL operation is critical in a wide range of applications.
We consider the following model of the temperature being controlled through a TCL:
\[
x_{t+1} = \alpha x_t + (1- \alpha) (\theta - \eta RP u_t) + w_t,
\]
which was originally developed by Mortensen and Haggerty \cite{Mortensen1988}.
Here, $x_t \in \mathbb{R}$ is the temperature of interest (e.g., indoor temperature, food temperature), $u_t \in \{0,1\}$ is an ON/OFF control input, and $w_t \in \mathbb{R}$ is a disturbance variable that takes into account environmental and/or human behavioral uncertainty.
Note also that $\alpha = \exp(-h/CR)$, where $C$ is the thermal capacitance (kWh/$^\circ$C), $R$ is the thermal resistance ($^\circ$C/kW) and $h$ is the time interval between stages $t$ and $t+1$. In addition, the parameter 
$P$ represents the range of energy transfer to or from the thermal mass (kW) and the coefficient $\eta$ is the control efficiency.
In our numerical experiments, the following parameters are used:
$R = 2$ $^\circ$C/kW, $C = 2$ kWh/$^\circ$C, $\theta = 32$ $^\circ$C, $h = 5/60$ hour, $P = 14$ kW, and $\eta = 0.7$.
We compute the distributionally robust safe sets and policies for $90$ minutes with 5-minute time intervals between two consecutive stages.
The desired safe set of temperature values is chosen as $A = [19, 22]$ ($^\circ$C).

We now assume that an empirical estimate of the first- and second-moments is the only available information about $w_t$ except its support.
Let $\bold{m}$ and $\bold{\Sigma}$ be the empirical mean and variance of $w_t$. As discussed in Section \ref{amb_moment}, it is reasonable to consider the following constraints:
$| \mathbb{E}_{\mu_t} [ w_t ]  - \bold{m} | \leq b$, and
$\mathbb{E}_{\mu_t} [ (w_t - \bold{m})^2 ] \leq c\bold{\Sigma}$, where
$b \geq 0$ and $c \geq 1$ are adjustable parameters depending on one's confidence in the estimate. We call them \emph{confidence parameters}.
We set $\bold{m} = 0$, $\bold{\Sigma} = 0.25^2$, and the disturbance's support $\mathcal{K} = [-\frac{1}{2}\sqrt{\bold{\Sigma}/12}, \frac{1}{2}\sqrt{\bold{\Sigma}/12}]$.

\begin{figure}
\begin{center}
\includegraphics[width=3.3in]{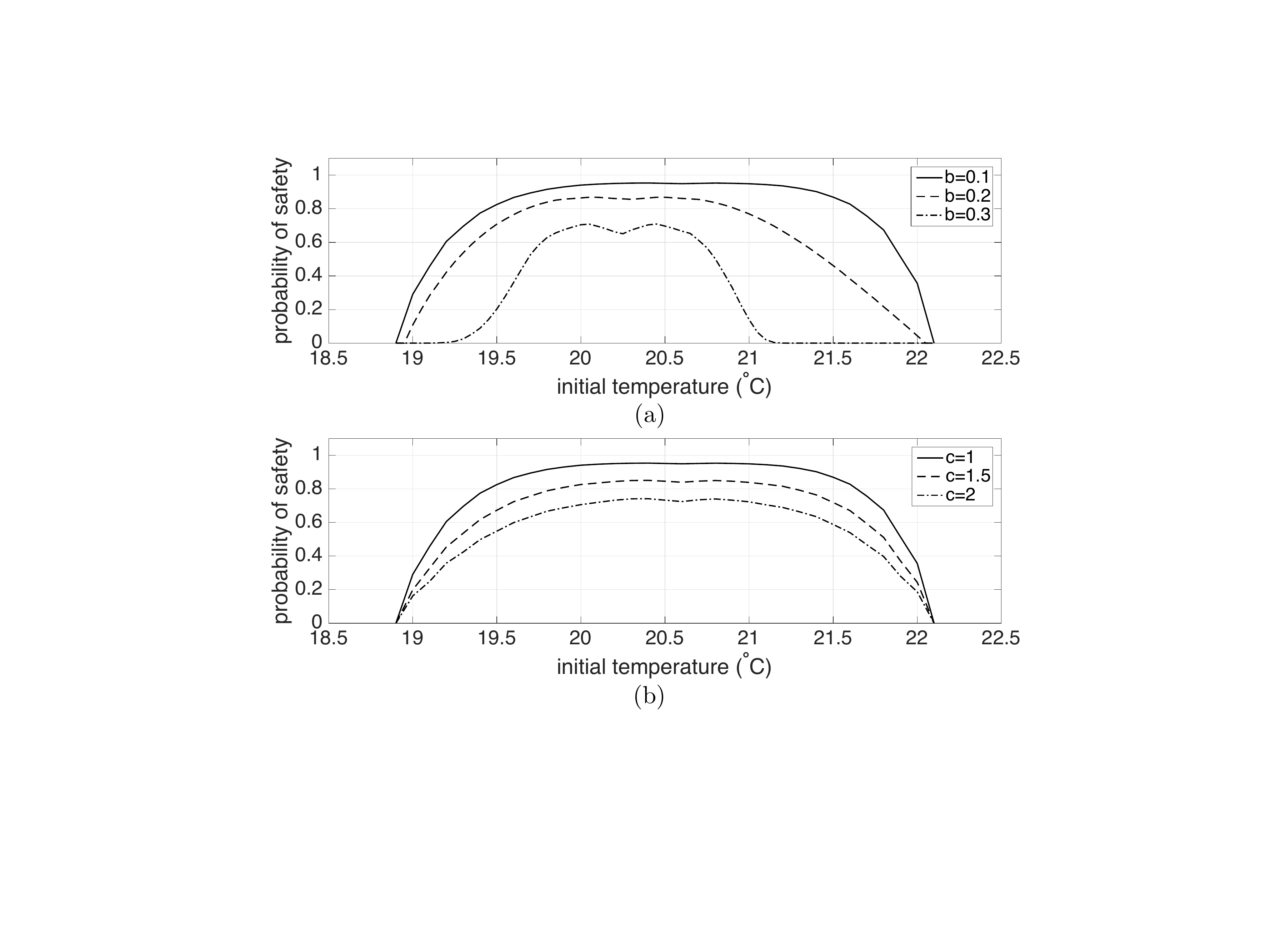}    % The printed column  
\caption{Effect of (a) the confidence parameter $b$ for the mean $\bold{m}$ and (b) the confidence parameter $c$ for the variance $\bold{\Sigma}$ on the probability of safety.}  
\label{fig:effect}                                 % Size the figures 
\end{center}                                 % accordingly.
\end{figure}

\subsection{Effect of the Confidence Parameters}

Fig. \ref{fig:effect} shows the probability $P_{\bm{x}}^{\mbox{\tiny safe}}$ of safety as a function of the initial state $\bm{x} \in [18, 23]$ for multiple confidence parameters $b$ and $c$.
The function has a bimodal structure due to the discrete ON/OFF control inputs:  it can be considered as the point-wise maximum of the probability function with OFF control and that with ON control.
At initial states between the two peaks,
ON or OFF control input at time $0$  does not maximize the probability of safety.
Aditionally, as the initial state approaches the boundary of the set $A = [19,22]$,
the probability of safety decreases.

As shown in Fig. \ref{fig:effect} (a), given $c = 1$, the probability of safety decreases with the confidence parameter $b$ for the mean.
In other words, an inaccurate mean estimate $\bold{m}$ makes it difficult for the system to remain safe for all stages. 
Fig. \ref{fig:effect} (b) illustrates that the probability of safety decreases
as the uncertainty of the variance estimate $\bold{\Sigma}$ increases when $b = 0.1$.
The probability $P_{\bm{x}}^{\mbox{\tiny safe}}$ of safety scales down with the variance confidence parameter $c$ without any change in its support. 
On the other hand, an increase in the mean confidence parameter $b$ reduces the support of $P_{\bm{x}}^{\mbox{\tiny safe}}$.
This is because a change in  $b$  may shift the worst-case disturbance distribution  while a change in $c$ can only scale the worst-case distribution.

\subsection{Safety-Oriented Distributionally Robust Control}

To demonstrate the performance of the proposed safety-oriented distributionally robust controller, 
we compare it to a safety-oriented controller synthesized with standard probabilistic safe sets.
We use the safety-oriented controller design approach proposed in Section \ref{socon}.
% to minimize the energy consumption while guaranteeing safety.
%\footnote{We omitted the results with a safety-preserving controller constructed using deterministic safe sets because these sets are empty given the problem data.  Thus, all the trajectories generated by using this controller violate the safety constraint.}
When the control action is allowed to be arbitrarily chosen in~\eqref{soc}, we choose OFF control input to minimize the energy cost. In other words, it is an energy cost-minimizing safety-oriented controller.
Suppose that the true disturbance distribution is uniformly distributed over the support $\mathcal{K}$ with mean $\bold{m}$, and variance $\bold{\Sigma}$.
We consider the situation in which we misestimate the distribution as
a truncated normal distribution with the same support $\mathcal{K}$, mean $\bold{m}$ and variance $\bold{\Sigma}/2$.\footnote{The variance of the truncated normal distribution must be greater than the variance of the uniform distribution with the same support and mean.}
We set the probability threshold as $\alpha = 0.95$ and 
construct the probabilistic safe sets using the method proposed by Abate et al. \cite{Abate2008} with the inaccurately estimated distribution.
As shown in Fig. \ref{fig:temp} (a), 
the safety-oriented controller obtained using the probabilistic safe sets 
fails to guarantee that  the probability of safety will be greater than or equal to the threshold $\alpha = 0.95$. Specifically, in our numerical experiment with the initial state $\bm{x}=21$, $1$,$365$ of $10$,$000$ sample trajectories violated the safety constraints. 
Thus, the probability of the system being safe for all stages is only $0.8635$ even though the proposed safety-oriented controller is constructed in a conservative way for safety.
However, when the distributionally robust safe sets are employed to construct the safety-oriented controllers, only $5$ sample trajectories move out from the set $A = [19, 22]$. 
In other words, the probability of safety is $0.995$.\footnote{The resulting probability of safety is significantly greater than the threshold $0.95$ because the uniform distribution is not the worst-possible distribution.}

\begin{figure}
\begin{center}
\includegraphics[width=3.3in]{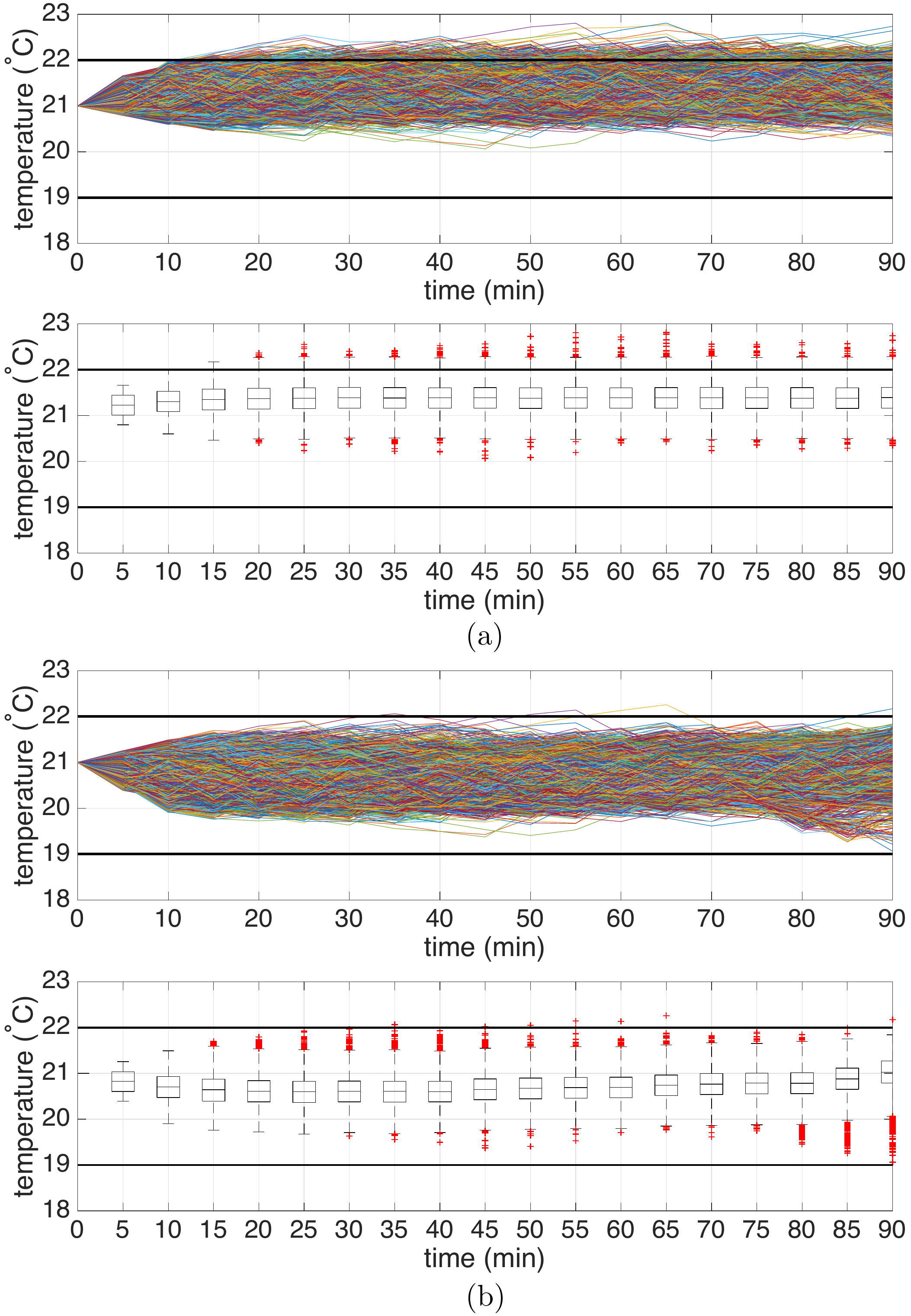}    % The printed column  
\caption{State trajectories and their Tukey box plot generated by the safety-oriented  controller obtained using (a) standard probabilistic safe sets  and (b) distributionally robust safe sets.
}  
\label{fig:temp}                                 % Size the figures 
\end{center}                                 % accordingly.
\end{figure}

\section{Conclusions}

We proposed a dynamic game approach to computing distributionally robust safe sets and policies concerning ambiguous information about the probability distribution of disturbances. 
We identified conditions under which a Markov policy is an optimal distributionally robust safe policy.
Such a policy leads to a practical design method for safety-oriented stochastic controllers that guarantee that the probability of the system being safe for all remaining stages exceeds a pre-specified threshold, regardless of how the disturbance distribution is chosen in an ambiguity set.
We also proved that there is no duality gap in the inner minimization problem of the Bellman equation when ambiguity sets with moment uncertainty are considered. 
This strong duality result allows us to reformulate the infinite-dimensional minimax problem in the Bellman equation as a semi-infinite program.
Since the reformulated dual Bellman equation can be solved using existing convergent algorithms, the proposed dual formulation method alleviates computational issues that otherwise occur when solving the distributionally robust safety specification problem.
Through numerical simulations, we  demonstrated that 
the safety-oriented controller constructed using the distributionally robust safe sets and policies can guarantee the desired probability of safety even when the probability distribution of disturbances is inaccurately estimated; meanwhile, the same controller based on standard probabilistic safe sets cannot.

The following future directions are of interests to generalize and improve the proposed approach. First,  this method is readily applicable to backward reachability analysis. Given a target set $B$, the probability that reaches $B$ for some $t \in \bar{\mathcal{T}}$ given the strategy pair $(\pi, \gamma)$ and the initial value $\bm{x}$ can be expressed as
$P_{\bm{x}}^{\mbox{\tiny reach}}(\pi, \gamma; B) = \mathbb{E}^{\pi, \gamma} [\max_{t \in \bar{\mathcal{T}}} \bold{1}_B(x_t) ]$. We can then use dynamic programming and the same dual formulation to obtain a control policy that maximizes this probability. Similarly, the proposed distributionally robust approach can be extended to reach-avoid problems.
Second, the proposed method is compatible with other types of ambiguity sets.
The existence result and safety-oriented controller design approach remain valid when different ambiguity sets are considered.
However, different reformulations of associated Bellman equations are required to handle other ambiguity sets such as statistical distance-based ones by using different existing strong duality results. 
Third, this method presents a scalability issue arising from both dynamic programming and semi-infinite programming.
Applying and developing advanced computational techniques including approximate dynamic programming, and sampling- and moment-based approximation methods  is of great interests to alleviate the scalability issue.

\bibliographystyle{plain}        % Include this if you use bibtex 
\bibliography{reference}

\appendix
\section{Proof of Proposition \ref{concave}}

We use mathematical induction to prove this proposition. 
For $t = T$, $v_T(\bm{x}) = \bold{1}_A(\bm{x})$ is concave with respect to $\bm{x} \in A$. 
Suppose that $\bm{x} \mapsto v_s(\bm{x})$ is concave for  $\bm{x} \in A$ for $s = T-1, \cdots, t+1$.
Fix arbitrary $\bm{x}^1, \bm{x}^2 \in \mathbb{R}^n$ and $\lambda \in (0,1)$. Let $\bm{u}^i \in \mathbb{U}(\bm{x}^i)$ be a solution to the outer maximization problem in \eqref{dp} at $(t, \bm{x}^i)$ for $i = 1,2$.
We also let $\bm{x}^\lambda := \lambda \bm{x}^1 + (1-\lambda) \bm{x}^2$ and $\bm{u}^\lambda := \lambda \bm{u}^1 + (1-\lambda) \bm{u}^2$.
Since $A$ is a convex set,
$\bm{x}^\lambda \in A$.
In addition,  Assumption~\ref{struct}~$(iii)$ implies that
$\bm{u}^\lambda \in \mathbb{U} (\bm{x}^\lambda)$.
Therefore, we have that
\begin{equation} \nonumber
\begin{split}
v_t(\bm{x}^\lambda) \geq \inf_{\bm{\mu} \in \mathbb{D}_t} \int_{\mathbb{R}^l} v_{t+1} (f(\bm{x}^\lambda, \bm{u}^\lambda, \bm{w})) d\bm{\mu}(\bm{w}).
\end{split}
\end{equation}
Since $f$ is an affine function, for each $\bm{w} \in \mathbb{R}^l$,
$f(\bm{x}^\lambda, \bm{u}^\lambda, \bm{w}) = 
\lambda f(\bm{x}^1, \bm{u}^1, \bm{w}) + (1-\lambda) f(\bm{x}^2, \bm{u}^2, \bm{w})$.
Combining this with the concavity of $v_{t+1}$, we further have that
\begin{equation}\nonumber
\begin{split}
v_t(\bm{x}^\lambda) &\geq \inf_{\bm{\mu} \in \mathbb{D}_t}  \int_{\mathbb{R}^l}  \lambda \big [ v_{t+1} (f(\bm{x}^1, \bm{u}^1, \bm{w})) \\
& \qquad \qquad\;+  (1-\lambda)  v_{t+1} (f(\bm{x}^2, \bm{u}^2, \bm{w})) \big ] d\bm{\mu}(\bm{w})\\
&\geq \lambda \inf_{\bm{\mu} \in \mathbb{D}_t}  \int_{\mathbb{R}^l} v_{t+1} (f(\bm{x}^1, \bm{u}^1, \bm{w})) d\bm{\mu}(\bm{w})\\
& +  (1-\lambda) \inf_{\bm{\mu} \in \mathbb{D}_t}  \int_{\mathbb{R}^l} v_{t+1} (f(\bm{x}^2, \bm{u}^2, \bm{w})) d\bm{\mu}(\bm{w})\\
&= \lambda v_{t+1} (\bm{x}^1) + (1-\lambda) v_{t+1} (\bm{x}^2),
\end{split}
\end{equation}
which implies that $v_t$ is concave. 
This completes our inductive argument.~\qed

\end{document}